\documentclass{article}
\pdfoutput=1
\usepackage{epsfig,fancyhdr,latexsym, cancel}
\usepackage[left=1in,right=1in,top=1in,bottom=1in]{geometry}
\usepackage{tabto}
\usepackage{array}%use to set spacing between rows of table in abstract page
\usepackage{float} %to force-position figures in text
\usepackage{pgfplots}
\pgfplotsset{compat = 1.18}
\usepackage{tikz-cd}
\usepackage{colordvi, comment, color}
\usepackage{caption, subcaption}%for subfigures
\usepackage{amsbsy,amsfonts,amsmath,amssymb,amsthm,amsxtra, amstext, mathtools}
\usepackage{algorithm, algpseudocode}

\usepackage[stable]{footmisc}
\usepackage{indentfirst} %used to indent all first para
\usepackage{verbatim}
\usepackage{mathtools}
\usepackage{enumerate}
\usepackage{url}
\usepackage[utf8]{inputenc} % Ensure this is included
\usepackage[T1]{fontenc}

\usepackage{color}
\definecolor{darkblue}{rgb}{0.0, 0.0, 0.55}
\definecolor{darkred}{rgb}{0.5, 0.0, 0.13}
\definecolor{darkgreen}{rgb}{0.0, 0.2, 0.13}

\usepackage{hyperref}%colors for hyperlinks in the document
\hypersetup{colorlinks=false, urlcolor=darkgreen, linkcolor=darkblue, citecolor=darkred}
\usetikzlibrary{3d, calc}%for making figures
\usepackage{tabularx}%for making tables
\usepackage{fancyhdr}
\usepackage{graphicx}
\usepackage{extarrows} %http://ctan.org/pkg/extarrows
\usepackage[T1]{fontenc}
\usepackage[style=ext-numeric, isbn=false, maxnames=99]{biblatex}
\DeclareFieldFormat{issuedate}{#1}

\AtEveryBibitem{\clearfield{month}}
\usepackage{longfbox}
\usepackage{tocloft} %table of contents
    \cftsetindents{section}{0em}{6.55em} %adjust sections title position
    \cftsetindents{figure}{0em}{0.6em} %adjust sections title position

\theoremstyle{plain}
\newtheorem{theorem}{Theorem}[section]
\newtheorem{corollary}[theorem]{Corollary}
\newtheorem{proposition}[theorem]{Proposition}
\newtheorem{lemma}[theorem]{Lemma}

\theoremstyle{definition}
\newtheorem{definition}[theorem]{Definition}

\theoremstyle{remark}

\def\bee{\begin{eqnarray}}
\def\bes{\begin{eqnarray*}}
\def\eee{\end{eqnarray}}
\def\ees{\end{eqnarray*}}

\def\dsp{\displaystyle}

\def\one{{\hbox{1{\kern -0.35em}1}}}

\def\a{\alpha}

\def\e{\epsilon}

\def\s{\sigma}
\def\t{\tau}

\def\(({\Big(}
\def\)){\Big)}
\def\({\bigl(}
\def\){\bigr)}

\def\]{\right]}
\def\[{\left[}

\newcommand{\R}{\mathbb{R}}

\newcommand{\X}{\mathbb{X}}

\DeclareUnicodeCharacter{0301}{\hspace{-1ex}\'{ }}

\title{Simplicial Hausdorff Distance for Topological Data Analysis}
\author{Nkechi Nnadi\footnotemark[1] \footnotemark[2] and Daniel Isaksen\footnotemark[1]
}
\date{July 2024}

\begin{document}
 \footnotetext[1]{Department of Mathematics, Wayne State University, Detroit, MI, USA} \footnotetext[2]{\textcolor{blue}{\emph{nkechinnadi@wayne.edu}} \textit{Corresponding author}}
\maketitle

\begin{abstract}
Many practical applications in topological data analysis arise from data in the form of point clouds, which then yield simplicial complexes. The combinatorial structure of simplicial complexes captures the topological relationships between the elements of the complex. In addition to the combinatorial structure, simplicial complexes possess a geometric realization that provides a concrete way to visualize the complex and understand its geometric properties. This work presents an amended Hausdorff distance as an extended metric that integrates geometric proximity with the topological features of simplicial complexes. We also present a version of the simplicial Hausdorff metric for filtered complexes and show results on its computational complexity. In addition, we discuss concerns about the monotonicity of the measurement functions involved in the setup of the simplicial complexes.
\end{abstract}

\section{Introduction}
 While the geometric realization of a (finite) simplicial complex provides insights into its geometric properties, the combinatorial structure captures its topological properties. However, such a geometric realization cannot always be embedded in $\R^d$ for some fixed $d>0$. The overlaps between faces of the simplicial complex form higher-dimensional simplices, which cannot be fully realized geometrically in $\R^d$. In order to geometrically realize an abstract simplicial complex $X$ without self-intersections in $\R^d$, it is necessary that $d$ be at least equal to the dimension of $X$. 
 
 For example, a $2$-dimensional simplicial complex can be embedded in $\R^2$ without self-intersections, while a $3$-dimensional complex requires embedding in $\R^3$. While we can represent the vertices and edges of a tetrahedron in $\R^2$, we cannot fully represent the three-dimensional structure without distorting it. A tetrahedron inherently has three dimensions, and trying to capture it in only two dimensions necessarily involves some loss of information.  Therefore, while it is possible to represent a tetrahedron in $\R^2$ pictorially, it is not a faithful representation of the original three-dimensional object.
 
 \begin{figure}[h!]
     \centering
     \includegraphics[width=0.3\linewidth]{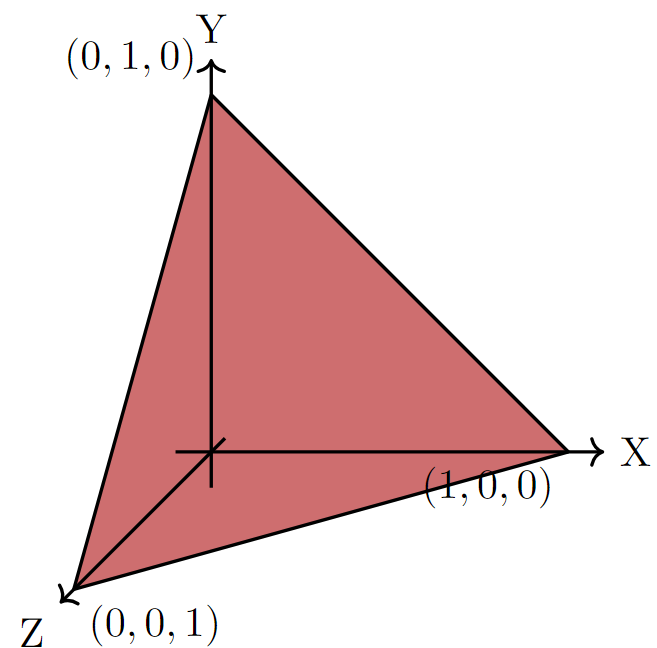}
     \caption{An embedding of a $2$-simplex is possible only for $\R^d$ with $d\geq 2$.}
     \label{fig1}
 \end{figure}
 
Lower-dimensional spaces cannot fully capture the geometric richness of higher-dimensional structure. This is because any $k$-dimensional simplices, with $k\geq d$, do not lie in $\R^d$. When higher-dimensional simplicial complexes are presented in lower dimensions, they may be difficult to recognize and interpret. This necessitates the shading scheme usually employed in figures of simplicial complexes to distinguish higher-order simplices from lower-order simplices.

%re-word this paragraph::
Additionally, in order to capture the geometric and topological features of simplicial complexes, we need to be able to distinguish different dimensions of simplices. 

\begin{figure}[h!]
    \centering
    \includegraphics[width=0.5\linewidth]{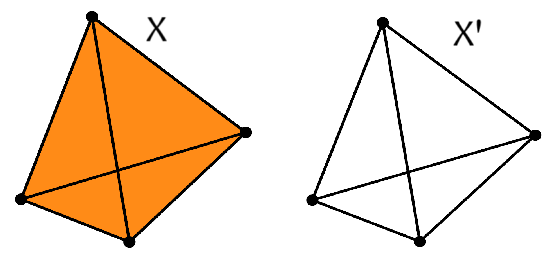}
    \caption{Can we tell that $X$ and $X'$ are different?}
    \label{fig2}
\end{figure}

Moreover, for most practical applications, there are often constraints on the number of dimensions that can be feasibly used due to hardware limitations, memory constraints, and the need for efficient algorithms. 
Hence, it is essential to strike a balance between the representational power of higher dimensions and the practical constraints of computational resources and data characteristics. The combinatorial nature and geometric realization of simplicial complexes inform their representation in lower-dimensional spaces. While geometric intuition is valuable, understanding the interplay between geometry and topology is essential for interpreting lower-dimensional embeddings and their implications for data analysis and visualization. For these reasons, we investigate a distance measurement between simplicial complexes that reflects both geometric and topological proximity. 

\section{Simplicial Hausdorff Distance}
\subsection{The Hausdorff distance}
The canonical choice of a distance to measure similarity between subsets of a metric space is the \emph{Hausdorff distance}. %Felix Hausdorff \cite{hausdorff} introduced the concept of the Hausdorff distance in 1914. 
The classical Hausdorff distance is described as follows \cite {hausdorff2}. Let $A$ and $B$ be two subsets of a metric space $X$. The \emph{directed Hausdorff distance} between $A$ and $B$ is defined as follows:
$$\vec{d}_H(A,B) =  \sup\limits_{a\in A} \inf\limits_{b\in B} d_X(a, b)$$
where $d_X(\cdot,\cdot)$ is the metric on $X$. It effectively illustrates the concept of how $B$ must be uniformly enlarged in all directions to encompass $A$. The \emph{Hausdorff} distance is defined as follows:
$$d_H(A,B)=\max\Bigl\{\vec{d}_H(A,B), \: \vec{d}_H(B,A)\Bigr\}.$$

The Hausdorff distance is a \emph{pseudo-metric} on the power set, $2^X$, of $X$. This means that it is non-negative, symmetric, and satisfies the triangle inequality; however, distinct sets may have their Hausdorff distance equal to zero. The Hausdorff distance is a metric when restricted to a collection $\mathcal{H}$, of non-empty compact subsets of $X$\cite{munkres}. 

The Hausdorff distance is a natural choice of a distance for compact subsets of $\R^d$%what does this mean "this project"?
because it is invariant under isometries (for instance, rotations and translations) of both sets. However, the Hausdorff distance is sensitive to the presence of outliers in data sets \cite{chazal}, and computationally expensive when applied to geometric simplices as compact subsets of Euclidean space.

\subsection{Definition of the \textit{simplicial} Hausdorff distance}
Consider the class, $\mathbb{X}^d$, of pairs $(X,f)$ where $X$ is a finite simplicial complex, most precisely a Vietoris-Rips complex, and $f$ is a one-to-one measurement function such that $f: X_0\rightarrow \mathbb{R}^d$ yields a point cloud, with $X_0$ as the set of vertices of $X$. %Let $d(\cdot,\cdot)$ denote the Euclidean distance in $\mathbb{R}^d$.

We introduce a concept of proximity between simplicial complexes $(X,f)$ and $(Y,g)$ in $\X^d$, known as $\e$-closeness. As a directed distance, $\e$-closeness  measures the maximal nearest neighbor distance of a vertex in $(X,f)$ from a vertex in $(Y,g)$ by sifting through all the participating simplices.

\begin{definition}[$\epsilon$-closeness] Let $(X,f)$ and $(Y,g)$ be from the class $\mathbb{X}^d$. The pair $(X,f)$ is \textcolor{darkred}{\emph{$\epsilon$-close}} to $(Y,g)$ if and only if: for every $k$-simplex $\sigma$ in $X$, there is a $k$-simplex $\tau$ in $Y$ such that for any vertex $v \in \sigma$, there is some vertex $w\in \tau$ with $d\left(f(v), g(w)\right) < \epsilon,$ for every $k$.
\end{definition}

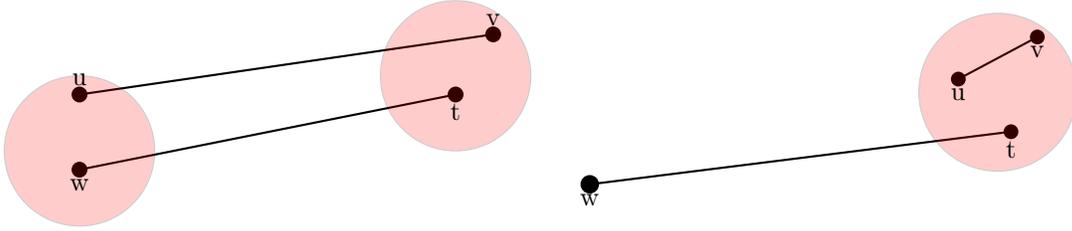
\begin{figure}[ht!]
\centering
\begin{minipage}{0.45\textwidth}
\begin{tikzpicture}

% Define the new radius of the balls
\def\radius{1}

% Define the coordinates of the vertices
\coordinate (w) at (0,0.5);
\coordinate (t) at (5,1.5);
\coordinate (u) at (0,1.5);
\coordinate (v) at (5.5,2.3);

% Draw the edges
\draw[thick] (w) -- (t);
\draw[thick] (u) -- (v);

% Draw the vertices
\fill[black] (w) circle (3pt);
\fill[black] (t) circle (3pt);
\fill[black] (u) circle (3pt);
\fill[black] (v) circle (3pt);

% Calculate the centers of the balls
\coordinate (CenterLeft) at (0, 0.75); % Midpoint between A1 and A2
\coordinate (CenterRight) at (5, 1.75); % Midpoint between B1 and B2

% Draw the transparent balls around the left and right vertices
\draw[fill=red, opacity=0.2] (CenterLeft) circle (\radius);
\draw[fill=red, opacity=0.2] (CenterRight) circle (\radius);

% Label the vertices
\node[below] at (w) {w};
\node[below] at (t) {t};
\node[above] at (u) {u};
\node[above] at (v) {v};

\end{tikzpicture}
\end{minipage}
\begin{minipage}{0.45\textwidth}
\begin{tikzpicture}[scale=0.7]

% Define the new radius of the balls
\def\radius{1.5}

% Define the coordinates of the vertices
\coordinate (w) at (2,0);
\coordinate (t) at (10,1); % Extended length
%\coordinate (z) at (0,2.5); % Adjusted height for clarity
\coordinate (u) at (9,2); % Extended length
\coordinate (v) at (10.5,2.8); % Closer to u

% Draw the edges
\draw[thick] (w) -- (t);
\draw[thick] (u) -- (v);

% Draw the vertices
\fill[black] (w) circle (5pt);
\fill[black] (t) circle (4pt);
%\fill[black] (z) circle (3pt);
\fill[black] (u) circle (4pt);
\fill[black] (v) circle (4pt);

% Calculate the centers of the balls
%\coordinate (CenterLeft) at (0, 1.25); % Midpoint between w and z
\coordinate (CenterRight) at (9.75, 1.75); % Midpoint to include t, u, and v

% Draw the transparent balls around the left and right vertices
%\draw[fill=blue,opacity=0.2] (CenterLeft) circle (\radius);
\draw[fill=red, opacity=0.2] (CenterRight) circle (\radius);

% Label the vertices
\node[below] at (w) {w};
\node[below] at (t) {t};
%\node[above] at (z) {z};
\node[below] at (u) {u};
\node[below] at (v) {v};

\end{tikzpicture}
\end{minipage}

\caption[Illustration of $\e$-closeness]{Illustration of $\e$-closeness: each orange ball has radius $\frac{\e}{2}$. In the figure on the left, the $1$-simplex $wt$ is $\e$-close to $uv$ because $d(u,w)$ and $d(v,t) < \e$; in the figure on the right, the $1$-simplex $uv$ is $\e$-close to $wt$ because $d(u,t)$ and $d(v,t) < \e$; but $wt$ is not $\e$-close to $uv$ because $d(w,u)> \e$ and $d(w,v) > \e$. The $\e$-closeness relation is not symmetric.}
\label{fig8}
\end{figure}

%\section{Metric Properties of the simplicial Hausdorff distance}
Next, we present the definition of the simplicial Hausdorff distance in terms of $\e$-closeness.

\begin{definition}[Simplicial Hausdorff distance]
The simplicial Hausdorff distance is a map $\delta:\mathbb{X}^d\times \mathbb{X}^d\rightarrow \mathbb{R}_{\geq 0}$ defined as:
\begin{eqnarray}\label{scd} 
   \nonumber \delta\Big((X,f),(Y,g)\Big) &= \max\Bigl\{ \vec{d}\left((X,f),(Y,g)\right), \vec{d}\left((Y,g), (X,f)\right)  \Bigr\}.
\end{eqnarray}
where the directed distance $\displaystyle \vec{d}\left((X,f),(Y,g)\right) = \inf\{\epsilon>0: (X,f) \text{ is }  \epsilon\text{-close to }(Y,g)\}$.
\end{definition}

\subsection{Metric Properties of the Simplicial Hausdorff distance}
This section aims to present interesting observations regarding the behavior of the simplicial Hausdorff distance.

\begin{proposition} \label{2-6}
The pair $(X,f)$ is $\epsilon$-close to $(Y,g)$ if and only if 
\begin{equation}
\max\limits_{k} \max\limits_{\substack{\sigma \in X\\ \dim{\sigma}=k} }\min\limits_{\substack{\tau \in Y\\ \dim{\tau}=k}} \max\limits_{v\in \sigma} \min\limits_{w \in \tau} d(f(v),g(w))\: < \: \epsilon
\end{equation}
\end{proposition}

\begin{proof}
$ $\newline
For brevity, write $X$ for the pair $(X,f)$ and $Y$ for $(Y,g)$. By definition, $X$ is $\e$-close to $Y$ when for each $k$-simplex $\s$ in $X$, there exists some $k$-simplex $\t$ in $Y$ such that for every vertex $v$ in $\s$, we can find some vertex $w$ in $\t$ with $d (f(v),g(w)) < \e $. This is equivalent to the statement: for each $k$-simplex $\s$ in $X$, there exists some $k$-simplex $\t$ in $Y$ such that for every vertex $v$ in $\s$, the minimum $\min\limits_{w\in\t} d(f(v),g(w)) < \e$. This, in turn, is equivalent to stating that for each $k$-simplex $\s$ in $X$, there exists some $k$-simplex $\t$ in $Y$ such that the maximum $\max\limits_{v\in\s} \min\limits_{w\in\t} d(f(v),g(w)) < \e$. This is equivalent to the statement: for each $k$-simplex $\s$ in $X$, the minimum $\min\limits_{\t\in Y} \max\limits_{v\in\s} \min\limits_{w\in\t} d(f(v),g(w)) < \e$. And this is in turn is equivalent to: $\max\limits_{k} \max\limits_{\substack{\sigma \in X\\ \dim{\sigma}=k} }\min\limits_{\substack{\tau \in Y\\ \dim{\tau}=k}} \max\limits_{v\in \sigma} \min\limits_{w \in \tau} d(f(v),g(w))\: < \: \epsilon$.
\end{proof}

Following proposition \ref{2-6}, we present an equivalent but more algorithmic definition of the simplicial Hausdorff distance. 

\begin{corollary}\label{cor1}
The simplicial Hausdorff distance may alternatively be defined as:
    \begin{equation} \label{scd2}
\nonumber \delta\Big((X,f),(Y,g)\Big) = \max \Bigl\{\vec{d}((X,f),(Y,g)), \: \vec{d}((Y,g), (X,f))\Bigr\},
\end{equation}
where $\vec{d}\left((X,f),(Y,g)\right)=\max\limits_{k} \max\limits_{\substack{\sigma \in X\\ \dim{\sigma}=k} }\min\limits_{\substack{\tau \in Y\\ \dim{\tau}=k}} \max\limits_{v\in \sigma} \min\limits_{w \in \tau} d(f(v),g(w))$, and $d(\cdot,\cdot)$ denotes Euclidean distance in $\mathbb{R}^d$. 
\end{corollary}

Next, we highlight the metric properties of the simplicial Hausdorff distance. To demonstrate the positive definiteness of the simplicial Hausdorff distance, it is necessary to show that the directed distances are well-behaved. This means that the least possible value of the directed distance from a $k$-simplex $X$ to another $k$-simplex $Y$ is zero, and this is achieved only when $X$ may be included in $Y$ similar to the notion of a subcomplex. Hence, we begin by putting some structure for hierarchy on the class $\X^d$.

\begin{definition}[Inclusion of members of $\X^d$]\label{incl}
Consider two pairs $(X,f)$ and $(Y,g)$, from the class $\X^d$. An inclusion of $(X,f)$ in $(Y,g)$ is a map $\phi: X\hookrightarrow Y$ satisfying:

\begin{enumerate}
\item There is an inclusion of the vertices of $X$ in the vertex set of $Y$. That is, $\dsp \phi_0: X_0 \hookrightarrow Y_0$.
\item $\phi$ induces an inclusion of the simplicial complex $X$ in $Y$. That is, $\phi(\sigma) \in Y$ is a $k$-simplex, for each $k$-simplex $\sigma \in X$\label{inclusion}.
\item The diagram below commutes:
\begin{equation*}    
    \begin{tikzcd}
    \centering
	\centering 
    {X_0} && {Y_0} \\
	& {\mathbb{R}^d}
	\arrow["{\phi_0}", from=1-1, to=1-3]
	\arrow["f"', from=1-1, to=2-2]
	\arrow["g", from=1-3, to=2-2]
    \end{tikzcd}
    \end{equation*}
\end{enumerate}
\end{definition}

\begin{definition}{Isomorphism in $\mathbb{X}^d$:}
Two pairs $(X,f), (Y,g) \: \in \mathbb{X}^d$ are \emph{isomorphic} if there exists a map $\displaystyle \phi: X \rightarrow Y$ such that $\phi$ satisfies the conditions of definition \ref{inclusion} and is bijective (as a simplicial map).
\end{definition}

In all the proofs that follow, write $X$ for the pair $(X,f)$ and $Y$ for $(Y,g)$ for the sake of brevity.

\begin{lemma}\label{2}
The directed distance, $\vec{d}\left((X,f), (Y,g)\right)$ is zero if and only if there exists an inclusion $ \displaystyle \phi: X \hookrightarrow Y$ of simplicial complexes in $\X^d$. 
\end{lemma}

\begin{proof}
$ $ \newline
Suppose $\vec{d}\left((X,f), (Y,g)\right) =0$. Then, proposition \ref{2-6} and corollary \ref{cor1} imply that for each $k$-simplex $\sigma \in (X,f)$, there exists a $k$-simplex $\tau \in (Y,g)$ with $\displaystyle \max_{v\in \sigma}\min_{w\in\tau} d\left(f(v),g(w)\right) = 0$. Then, for each vertex $v \in \sigma$, there exists a vertex $w\in \tau$ such that $\displaystyle f(v)=g(w).$ This implies that
$\displaystyle f(X_0) \subseteq g(Y_0)$. Since $f$ and $g$ are one-to one, we note that $X_0 \cong \text{image}(f)$ and $Y_0 \cong \text{image}(g)$. Hence we can define a map on the vertex sets  
    \begin{align*}
        \phi_0: X_0& \rightarrow Y_0\\
        f(v) &= g(\phi_0(v))
    \end{align*}
to be an inclusion of $X_0$ in $Y_0$.  

We extend this definition of vertex inclusions to higher-order simplices in the following manner. Supposing still that the directed distance, $\vec{d}\left((X,f), (Y,g)\right) = 0$, then for any $k$-simplex $\sigma \in X$, we can find another $k$-simplex $\tau \in Y$ such that each vertex $v\in\sigma$ is associated with some other vertex $w\in \tau$ so that $f(v)=g(w)$. That is, for any $k$-simplex $\sigma \in X$, we can find another $k$-simplex $\tau \in Y$ such that for every vertex $v \in \sigma$, there is a vertex $w\in\tau$ with $\phi_0(v)=w.$
It follows that for any $k$-simplex $\sigma \in X$, we can find another $k$-simplex $\tau \in Y$ such that $\phi_0(\sigma) = \tau.$ Therefore, $\phi: X \rightarrow Y$ defined in this manner is an inclusion.

Conversely, suppose $\phi: X \rightarrow Y$ is an inclusion of simplicial complexes in the sense of definition \ref{incl}. This means that $f=g \circ \phi_0$, and for every $k$-simplex $\s \in X$, there is another $k$-simplex $\t \in Y$ with $\phi(\s)=\t$. Then $\vec{d}\left((X,f), (Y,g)\right) = \max\limits_{k} \max\limits_{\substack{\sigma \in X\\ \dim{\sigma}=k} }\min\limits_{\substack{\tau \in Y\\ \dim{\tau}=k}} \max\limits_{v\in \sigma} \min\limits_{w \in \tau} d(f(v),g(w))$ is equal to $\max\limits_{k} \max\limits_{\substack{\sigma \in X\\ \dim{\sigma}=k} }\min\limits_{\substack{\tau \in Y\\ \dim{\tau}=k}} \max\limits_{v\in \sigma} \min\limits_{w \in \tau} d(g \circ \phi_0(v),g(w))$. By property $(iii)$ of definition \ref{incl}, it follows that $g \circ \phi_0(v) = g(w)$  for $v\in \sigma$ and $w \in \tau$. Therefore the directed distance $\vec{d}\left((X,f), (Y,g)\right)$ is zero.
\end{proof}

\begin{proposition}\label{main1}
$\delta$ is positive-definite. That is, $\delta\Big((X,f),(Y,g)\Big) \geq 0$ for all pairs $(X,f)$ and $(Y,g)$ in $\X^d$. Moreover, $\delta\Big((X,f),(Y,g)\Big) =0$ if and only if $(X,f)\cong (Y,g)$.
\end{proposition}

\begin{proof}
$ $\newline
Observe that $\displaystyle \delta\Big((X,f),(Y,g)\Big) = 0$ 
if and only if $\vec{d}((X,f),(Y,g)) = \vec{d}((Y,g), (X,f)) \;= \:0$. By lemma \ref{2}, this means that there exist simplicial inclusion maps $X\xhookrightarrow{\phi} Y$ and $Y\xhookrightarrow{\varphi} X$. These two inclusions imply, without loss of generality, that $\phi = \varphi^{-1}$, that $\phi\circ \varphi = \varphi \circ \phi$ is the identity map. This means, in particular, that $f(X_0) \cong g(Y_0)$. Hence, $(X,f) \cong (Y,g)$. Also, since the Euclidean metric is such that $d(f(v),g(w)) \geq 0$ for any vertices $v \in X$ and $w\in Y$, it follows that the directed distance, 
$ \dsp \vec{d}((X,f),(Y,g)) = \max\limits_{k} \max\limits_{\substack{\sigma \in X\\ \dim{\sigma}=k} }\min\limits_{\substack{\tau \in Y\\ \dim{\tau}=k}} \max\limits_{v\in \sigma} \min\limits_{w \in \tau} d(f(v),g(w))$ is non-negative. Therefore, the simplicial Hausdorff distance is positive definite. That is,
\bes
\delta\Big((X,f),(Y,g)\Big) = \max \Bigl\{\vec{d}((X,f),(Y,g)), \: \vec{d}((Y,g), (X,f))\Bigr\} \geq 0.
\ees
\end{proof}

\begin{proposition}\label{main2}
$\delta$ is symmetric. 
\end{proposition}

\begin{proof}
$ $\newline
The symmetry axiom is satisfied since the maximum member of a (finite) set is agnostic of the order of the elements in the set. That is,
\begin{eqnarray*}
    \delta\Big((X,f),(Y,g)\Big) &= \max \Bigl\{\vec{d}((X,f),(Y,g)), \: \vec{d}((Y,g), (X,f))\Bigr\}\\
    &=\max \Bigl\{\vec{d}((Y,g), (X,f)), \: \vec{d}((X,f),(Y,g))\Bigr\}.
\end{eqnarray*}
\end{proof}

To show the triangle inequality, we need other definitions and results.

\begin{definition}
Let $(X,f)$ and $(Y,g)$ be from the class $\X^d$. We say that a $k$-simplex $\sigma$ of $(X,f)$ is $\epsilon$-close to another $k$-simplex $\tau$ of $(Y,g)$ if the condition holds that for every vertex $v \in \sigma$, there exists a vertex $w\in \tau$ such that $d\left(f(v),g(w)\right)<\epsilon.$
\end{definition}

\begin{lemma}\label{e}
Consider three pairs $(X,f), \: (Y,g)$ and $(Z,h)$ from the class $\mathbb{X}^d$.  If $\sigma$, a $k$-simplex of $X$, is $\e$-close to $\t$ a $k$-simplex of $Y$ and $\t$ is $\e'$-close to another $k$-simplex, $\mu$, of $Z$, then $\s$ is $(\e+\e')$-close to $\mu$.
\end{lemma}

\begin{proof}
$ $ \newline
Given $\s, \: \t$ and $\mu$ satisfying the above hypothesis, then for any vertex $v \in \s$, there exists a vertex $w\in\t$ with $d\left(f(v),g(w)\right)<\epsilon$. For this $w\in\t$, there is a $u\in\mu$ with $d\left(g(w),h(u)\right)<\e'.$ By the triangle inequality property of the Euclidean metric, it follows that 
$$d(f(v),h(u)) \leq  d\left(f(v),g(w)\right) + d\left(g(w),h(u)\right)<\epsilon + \epsilon'$$
which proves that $\s$ is $(\e+\e')$-close to $\mu$.
\end{proof}

\begin{lemma}\label{ec}
 If $X$ is $\e$-close to $Y$ and $Y$ is $\e'$-close to $Z$, then $X$ is $(\e+\e')$-close to $Z$.
\end{lemma}

\begin{proof}
$ $ \newline
By definition, that $X$ is $(\e+\e')$-close to $Z$ means that for each $k$-simplex $\s\in X$, there exists a $k$-simplex $\mu$ of $Z$ such that $\s$ is $(\e+\e')$-close to $\mu$. Now, given any $k$-simplex $\s \in X$, there exists a $k$-simplex $\t \in Y$ such that $\s$ is $\e$-close to $\t$, and there exists a $k$-simplex $\mu \in Z$ such that $\t$ is $\e'$-close to $\mu$. By lemma \ref{e}, it follows that $\s$ is $(\e+\e')$-close to $\mu$. Therefore $X$ is $(\e+\e')$-close to $Z$.
\end{proof}

To prove the next proposition, we need the following standard result from real analysis \cite{baby1} about the infimum of subsets of real numbers:

\begin{theorem}\label{inf}
Let $A$ and $B$ be non-empty bounded sets in $\R$. Then, $\inf (A+B)= \inf A + \inf B.$
\end{theorem}

\begin{proposition}\label{main3}
$\delta$ satisfies the triangle inequality.
\end{proposition}

\begin{proof}
$ $\newline
Consider any three pairs $(X,f), \: (Y,g)$ and $(Z,h)$ from the class $\mathbb{X}^d$ satisfying the conditions of lemma \ref{ec}. We define the following sets $S:= \Bigl\{\epsilon>0: X \text{ is }\epsilon\text{-close to } Y\Bigr\}$, $T:= \Bigl\{\lambda>0: Y \text{ is }\lambda\text{-close to } Z\Bigr\}$, and $V:=\Bigl\{\varepsilon>0: X \text{ is }\varepsilon\text{-close to } Z\Bigr\}$. Observe that $S$, $T$, and $V$ are non-empty bounded subsets of $\R$. Let $S + T$ be defined to be $\Bigl\{ s+t: s\in S, t\in T\Bigl\}$. Then, lemma \ref{ec} implies that $S + T \subseteq V$. Take the infimum of both sides of the subset relation and use theorem \ref{inf}, to obtain the reverse-direction inequality $\inf V \leq \inf S + \inf T$. That is, $\vec{d}((X,f),(Z,h)) = \inf V$ is bounded above by $\vec{d}((X,f),(Y,g)) + \vec{d}((Y,g),(Z,h))$, which is, in turn, less than or equal to the sum $\delta\Big((X,f),(Y,g)\Big) + \delta\Big((Y,g),(Z,h)\Big)$ by the definition of $\delta$ as a maximum. 

By a similar argument, it can be shown that the directed distance $\vec{d}((Z,h),(X,f)) \leq \delta\Big((X,f),(Y,g)\Big) + \delta\Big((Y,g),(Z,h)\Big)$. Since each distance $\vec{d}((X,f),(Z,h))$ and $\vec{d}((Z,h),(X,f))$ is bounded above by the sum $\delta\Big((X,f),(Y,g)\Big) + \delta\Big((Y,g),(Z,h)\Big)$, then so is their maximum. Hence, it follows that
$$\delta\Big((X,f),(Z,h)\Big)\leq \delta\Big((X,f),(Y,g)\Big) + \delta\Big((Y,g),(Z,h)\Big).$$
\end{proof}

Now the main theorem of this chapter states:

\begin{theorem}\label{mainthm}
$\delta$ is an \textit{extended} metric.
\end{theorem}

\begin{proof}
The above propositions \ref{main1}, \ref{main2}, and \ref{main3} altogether prove the metric axioms. Note also, that $\delta$ can have an infinite value in the case where $(X,f)$ and $(Y,g)$ have different dimensions, since the maximum/minimum value of the empty set is infinite. Since $\delta$ takes on values in $[0,\infty) \cup \{-\infty,\infty\}$, it is an \emph{extended} metric.
\end{proof}

Next, we present results about the behavior of the simplicial Hausdorff distance for some special cases of simplicial complexes. We consider $X$ as a single vertex, then $X$ as a discrete set of vertices, then $X$ and $Y$ as both discrete, and finally, when $Y$ is a standard $n$-simplex, with $n>0$.

\begin{proposition}\label{prop1}
If $(X,f)$ and $(Y,g)$ are both discrete, then
$$\dsp \vec{d}\left((X,f),(Y,g)\right)= \vec{d}_H(f(X),g(Y)).$$
\end{proposition}

\begin{proof}
$ $\newline
By definition, $\vec{d}\left((X,f),(Y,g)\right)=\max\limits_{k} \max\limits_{\substack{\sigma \in X\\ \dim{\sigma}=k} }\min\limits_{\substack{\tau \in Y\\ \dim{\tau}=k}}\max\limits_{v\in\sigma}\min\limits_{w\in\tau}d\left(f(v),g(w)\right)$. Since $X$ and $Y$ both only consist of vertices, there is only the case $k=0$ which means $\vec{d}\left((X,f),(Y,g)\right)=\max\limits_{v \in X_0}\min\limits_{w \in Y_0}d\left(f(v),g(w)\right)$, which is equal to $\sup\limits_{v \in X}\inf\limits_{w \in Y}d\left(f(v),g(w)\right)$ with $X$ and $Y$ as finite and discrete sets. This may be re-written as $\sup\limits_{v \in X}\inf\limits_{w \in Y} \lVert f(v)-g(w)\rVert$, with $\lVert \cdot \rVert$ being the norm induced by the Euclidean metric, which is exactly $\vec{d}_H(f(X), g(Y))$ with the directed Hausdorff distance in the classical sense.

\end{proof}

\begin{proposition}[Stability]
 For $(X,f)$ in $\mathbb{X}^d$, $\delta\Big((X,f),(X,g)\Big) \leq \lVert f-g\rVert_{\infty}$ for a perturbation $g$ of $f$. 
\end{proposition}

\begin{proof}
    By definition, $\vec{d}\left((X,f),(Y,g)\right)=\max\limits_{k} \max\limits_{\substack{\sigma \in X\\ \dim{\sigma}=k} }\min\limits_{\substack{\tau \in Y\\ \dim{\tau}=k}}\max\limits_{v\in\sigma}\min\limits_{w\in\tau}d\left(f(v),g(w)\right)$. Noting that higher-dimensional simplices have larger number of participating vertices, and that each minimum in the expression is bounded above by the maximum of the same set, we obtain this upper bound for the directed distance: $\max\limits_{k} \max\limits_{\substack{\sigma \in X\\ \dim{\sigma}=k} }\max\limits_{\substack{\tau \in Y\\ \dim{\tau}=k}}\max\limits_{v\in\sigma}\max\limits_{w\in\tau}d\left(f(v),g(w)\right)$. The highest dimensional simplex in $X$ is one of its subcomplexes, and this yields a better upper bound on the directed distance, $\max\limits_{v\in X_0}$, which can be re-written as $\sup\limits_{v\in X_0}\lvert f(v)-g(v)\rvert$, which is equal to $\lVert f-g\rVert_{\infty}$.
\end{proof}

\subsection{Computational complexity of the simplicial complex distance}\label{complexity}
Recall that $\delta$ is defined as $ \dsp \delta\Big((X,f),(Y,g)\Big) = \max \Bigl\{\vec{d}((X,f),(Y,g)), \: \vec{d}((Y,g), (X,f))\Bigr\}$, where the directed distance $\vec{d}\left((X,f),(Y,g)\right)$ is equal to $\max\limits_{\sigma \in X} \min\limits_{\tau \in Y} \max\limits_{v\in \sigma} \min\limits_{w \in \tau} d(f(v),g(w))$. This algorithm (presented and discussed in chapter $4$) involves a sequence of steps.  Suppose $X_0$ has $n$ vertices and $Y_0$ has $m$ vertices. Let $\lvert X\rvert$ and $\lvert Y\rvert$ denote the dimensions of the simplicial complexes $(X,f)$ and $(Y,g)$, respectively. Then $(X,f)$ can consist of at most $n^{\lvert X\rvert +1}$ simplices. Similarly, $(Y,g)$ has a maximum of $m^{\lvert Y\rvert+1}$ simplices. 

\begin{proposition}
Denote $\displaystyle D=\max\{\lvert X \rvert, \lvert Y\rvert\}$, and $p = \max\{n,m\}$. Then, the simplicial Hausdorff distance, $\delta$, runs at most in polynomial time, $O(p^{2D+2})$.
\end{proposition}

\begin{proof}
$ $\newline
First, compute all pairwise Euclidean distances between vertices of $(X,f)$ and $(Y,g)$. Since $(X,f)$ and $(Y,g)$ have only a finite number of vertices, these pairwise distances may be stored in an array. This straightforward algorithm can be achieved in quadratic time, $O(nm)$. For the directed distances, computing the sequence of maxima across all simplices in $(X,f)$, then minimum across all simplices in $(Y,g)$, then maximum over all vertices of a simplex in $(X,f)$ and finally a minimum over vertices of a simplex in $(Y,g)$ will require $ \dsp \left(n^{\lvert X\rvert +1} \right) \cdot \left( m^{\lvert Y \rvert + 1} \right) \cdot \left( \lvert X\rvert +1 \right) \cdot \left( \lvert Y\rvert +1 \right)$ calculations. This approximates to $ \dsp \lvert X\rvert\cdot\lvert Y\rvert\cdot\left(n^{\lvert X\rvert +1}\right)\cdot\left(m^{\lvert Y\rvert+1}\right)$ calculations ignoring lower-order terms. This in turn approximates to $\left(n^{\lvert X\rvert +1}\right)\cdot\left(m^{\lvert Y\rvert+1}\right)$ calculations, up to dilation factor. Finding the maximum of a set with just two elements is feasible in constant time. Therefore, if we denote $\displaystyle D=\max\{\lvert X \rvert, \lvert Y\rvert\}$, and $p = \max\{n,m\}$. Then, the simplicial Hausdorff distance, $\delta$, runs at most in polynomial time, $O(p^{2D+2})$.  
\end{proof}

While higher-dimensional homology has theoretical importance in algebraic topology, practical considerations often limit its relevance in applied topological data analysis research. In many real-world datasets, the topological features of interest are often captured within the first few dimensions of homology. For example, in point cloud data representing a surface, the presence of connected components ($0$-dimensional homology) and loops ($1$-dimensional homology) may be more relevant than higher-dimensional voids or cavities. Computing higher-dimensional homology becomes exponentially more complex as the dimension increases.Beyond dimension $2$, the topological features become increasingly abstract and difficult to interpret in the context of the original data. This reduces the practical utility of higher-dimensional homology for most applications in topological data analysis. As the dimension increases, the effects of noise and sampling become more pronounced, making it challenging to extract meaningful topological features beyond a certain point. Higher-dimensional homology may be more susceptible to noise and may not provide additional insight into the underlying structure of the data. So, practical computations become infeasible beyond a certain dimension, typically around $2$, especially when dealing with large datasets.  Consequently, we can impose an upper bound on the dimensions of $(X,f)$ and $(Y,g)$ without loss of generality, as in the following corollary.

\begin{corollary}\label{cor2} With an upper bound of $2$ on $\lvert X \rvert$ and $\lvert Y\rvert$, the computational complexity improves to $O(p^6)$, where $p$ is the maximum number of vertices in either $X$ or $Y$. 
\end{corollary}

Many researchers consider a polynomial time complexity to indicate that the underlying problem is tractable. 

By comparison, in \cite{wang1}, the authors show that the Hausdorff distance between two polygons can be determined in time $O(n\log n)$, where $n$ is the total number of vertices. In \cite{altbook}, the authors show that finding the classical Hausdorff distance between two discrete sets, such as $X_0$ and $Y_0$, in $\R^2$ with $n$ and $m$ elements, can be done in $O((n+m)\log{(n+m)})$ time by use of Voronoi diagrams. This is an improved algorithm as compared with the straightforward algorithm, which has a running time of $O(nm)$. As we have discussed above, calculating the Hausdorff distance is feasible within polynomial time if the two sets can be expressed as simplicial complexes of a fixed dimension. In his doctoral thesis, Godau \cite{godau} demonstrates a method due to Alt et al \cite {altbook} for computing the directed Hausdorff distance between two sets in $\R^d$, each comprised of $n$ and $m$ $k$-dimensional simplices in polynomial time $O(nm^{k+2})$, for constant $k$ and $d$. A collection of $k$-dimensional simplices is not necessarily a simplicial complex. However, taking into account the above discussion on the restriction of the dimension of simplicial complexes, with $k=2$ and $p=\max{\{n,m\}}$, this running time simplifies to $O(p^5)$. Understandably, a slight increase in computational complexity is expected with different dimensions of simplices present in a simplicial complex and without the Voronoi construction as in corollary \ref{cor2}. In theorem $11$ of \cite{hausdorff2}, the authors show that the Hausdorff and directed Hausdorff distances for semi-algebraic sets belong to a class of much higher complexity than $NP$-hardness in the worst-case scenario.

\subsection{Simplicial Hausdorff distance for filtered complexes}
An analogous definition can be made for a filtration of simplicial complexes in the following manner. Consider a filtered complex $(X,f)$ such that for each $\a>0$, $X_{\a}$ is a simplicial subcomplex on the vertex set $X_0$ with index value $\a$ and $f: X_0 \rightarrow \R^d$ a measurement function on the vertex set. Also, $X_{\a} \leq X_{\beta}$ if $\a \leq \beta$. Note that for each $\a>0$, the simplicial complex $X_{\a}$ is in $\X^d$.

\begin{definition}
The simplicial Hausdorff distance between two filtered complexes $(X,f)$ and $(Y,g)$ is defined to be
\begin{eqnarray}\label{filtdist}
\hat{\delta}\Big((X,f),(Y,g)\Big) = \max \Bigl\{\vec{\mathsf{d}}((X,f),(Y,g)), \: \vec{\mathsf{d}}((Y,g), (X,f))\Bigr\}\\
\nonumber \text{ with } \vec{\mathsf{d}}\left((X,f),(Y,g)\right) = \sup\limits_{\a>0} \vec{d}\left((X_{\a},f),(Y_{\a},g)\right)
\end{eqnarray}
\end{definition}

This distance function $\hat{\delta}$ enjoys the same nice properties as the single complex version.

\begin{proposition}\label{filt1}
The simplicial Hausdorff distance for filtered complexes, $\hat{\delta}$, is positive-definite.
\end{proposition}

\begin{proof}
$ $\newline
By definition, for each $\a>0$, each directed distance $\vec{d}((X_{\a},f),(Y_{\a},g))$ and $\vec{d}((Y_{\a},g),(X_{\a},f))$ is at least $0$. By definition as suprema, it follows that the directed distances for the filtered complexes $\vec{\mathsf{d}}((X,f),(Y,g))$ and $\vec{\mathsf{d}}((Y,g), (X,f))$ are non-negative. Hence, their maximum $\hat{\delta}\Big((X,f),(Y,g)\Big)$ is also non-negative. Moreover, $\hat{\delta}\Big((X,f),(Y,g)\Big) = 0$ if and only if $\vec{\mathsf{d}}((X,f),(Y,g)) = \vec{\mathsf{d}}((Y,g), (X,f))=0$, which is true if and only if $\vec{d}((X_{\a},f),(Y_{\a},g)) = \vec{d}((Y_{\a},g),(X_{\a},f)) =0$ for each $\a>0$. By lemma \ref{main1}, this means $(X_{\a},f) \cong (Y_{\a},g)$ for each $\a>0$. Therefore the filtered complexes $(X,f)$ and $(Y,g)$ are isomorphic.
\end{proof}

\begin{proposition}\label{filt2}
The simplicial Hausdorff distance for filtered complexes, $\hat{\delta}$, as defined in equation \ref{filtdist}, is symmetric.
\end{proposition}

\begin{proof}
$ $\newline
The symmetry follows by its definition as a maximum of a finite set.
\end{proof}

\begin{proposition}\label{filt3}
The simplicial Hausdorff distance for filtered complexes, $\hat{\delta}$, follows the triangle inequality.
\end{proposition}

\begin{proof}
$ $\newline
Suppose we have three filtered complexes $(X,f), \: (Y,g)$ and $(Z,h)$, such that for each index $\a>0$, the pairs $(X_{\a},f), \: (Y_{\a},g)$ and $(Z_{\a},h)$ are in the class $\X^d$. We know from lemma \ref{main3} that $\vec{d}((X_{\a},f),(Z_{\a},h)) \leq \vec{d}((X_{\a},f),(Y_{\a},g)) + \vec{d}((Y_{\a},g),(Z_{\a},h))$, for each index $\a>0$. Evaluating for a supremum over all $\a>0$ preserves the inequality and yields: $\vec{\mathsf{d}}((X,f),(Z,h)) \leq \vec{\mathsf{d}}((X,f),(Y,g)) + \vec{\mathsf{d}}((Y,g),(Z,h))$. Since $\hat{\delta}$ is a maximum, this, in turn, implies that $\vec{\mathsf{d}}((X,f),(Z,h))  \leq \hat{\delta}\Big(((X,f),(Y,g)\Big) + \hat{\delta}\Big((Y,g),(Z,h)\Big)$. By a similar argument, we obtain that $\vec{\mathsf{d}}((Z,h),(X,f))  \leq \hat{\delta}\Big(((X,f),(Y,g)\Big) + \hat{\delta}\Big((Y,g),(Z,h)\Big)$. Consequently, the maximum of $\vec{\mathsf{d}}((X,f),(Z,h))$ and $\vec{\mathsf{d}}((Z,h),(X,f))$ is also bounded above by $\hat{\delta}\Big(((X,f),(Y,g)\Big) + \hat{\delta}\Big((Y,g),(Z,h)\Big)$. Hence,
$$\hat{\delta}\Big(((X,f),(Z,h)\Big) \leq \hat{\delta}\Big(((X,f),(Y,g)\Big) + \hat{\delta}\Big((Y,g),(Z,h)\Big).$$
\end{proof}

Thus we have proved that:

\begin{theorem}
The simplicial Hausdorff distance for filtered complexes, $\hat{\delta}$, defined in equation \ref{filtdist}, is an \textit{extended} metric. 
\end{theorem}

\begin{proof}
As demonstrated in propositions \ref{filt1}, \ref{filt2} and \ref{filt3}, $\hat{\delta}$ satisfies all the axioms for a metric. It will take on infinite values if there exists $\a$ for which $X_{\a}$ and $Y_{\a}$ have different dimensions.
\end{proof}

\section{Monotonicity of measurement functions}\label{mono}

Recall that the simplicial Hausdorff distance is defined on the class of pairs
\begin{eqnarray} \displaystyle 
    \nonumber \mathbb{X}^d := \: \biggl\{(X,f): &X \text{ is an abstract simplicial complex,} \text { and }  f \text{ is a one-to-one measurement} \\
   \nonumber &\text{function with } f:X_0\rightarrow \mathbb{R}^d, \text { and } X_0 \text{ is the set of vertices of } X \biggr\}\end{eqnarray}
 In the definition of the simplicial Hausdorff distance, we require the functions $f$ and $g$ to be one-to-one, and we discuss the motivations for this hypothesis here.

  %Check these facts and elaborate!!!

In practice, consider the sets of vertices $X_0 \leq X$ and $Y_0\leq Y$ to be similar to sample spaces. For example, these can represent groups of people in a certain population. The functions $f$ and $g$ are measurements with output in $\mathbb{R}^d$ (where $d$ is the number of different measurements taken, for instance, including height, weight, red blood cell count, cholesterol level, fingernail length, etc). In such a scenario, it is highly unlikely that any given two (or more) people will have the exact same readings for all $k$ measurements.
 
We also assume that the functions $f$ and $g$ are floating point-valued (this is different from being real-valued in that the set of floating points is discrete). If the measurements were all integer-valued with a short range, there would possibly not be much room for variability. However, if $f$ and $g$ are floating point-valued in at least one coordinate, then $f$ and $g$ are almost surely one-to-one. For example, there will likely be repeated scores if students are graded on a test over 20. Note, however, that this is for only one of the $d$ coordinates. Measurements in the other $d-1$ coordinates will likely differ for any two students who have the same test grade. So, the type and measurement resolution of the data may also influence the monotonicity of the $f$ and $g$. This fact of \emph{general position} is one reason why we can hypothesize that $f$ and $g$ are one-to-one without loss of generality. Almost every dataset of interest is such that each data point in $X_0$ can be identified as one unique vertex of the simplicial complex $X$.

\subsection{When there are repeats in node coordinates of measurement functions}\label{2.6.1}
Now, consider the highly unlikely scenario where there are actually repeats in the output of $f$ and $g$. The concern should be whether or not this will change the topology of the underlying sample space. The answer is no, and we discuss why in this section. We show that a Vietoris-Rips complex featuring these clusters of repeats is homotopy equivalent to a Vietoris-Rips complex in which every cluster of coinciding vertices is treated as a single vertex. Consequently, this is another reason we can hypothesize that $f$ and $g$ are injective, without loss of generality\footnote[1]{The repeats do matter, however, for statistical purposes. Consider, for example, that the mean of $2$ and $3$ differs from the mean of $2$, $3$ and $3$.}.

%% See the text Computational Topology: An Introduction by G¨unter Rote and Gert Vegter in "TDA Texts" folder for discussion on simplicial collapses. April 25, 2024

The \emph{link} of a vertex is a concept that essentially captures the local structure around a particular vertex within the larger simplicial complex. The precise definition is as follows:

\begin{definition}[Link of a vertex]
Let $K$ be an abstract simplicial complex, and $v$ be a vertex in $K$. The link of a vertex $v \in K$ is the subcomplex of $K$ described as:
$$\text{link}_K(v) = \bigl\{\s \in K: v \notin \s \text { and } v \cup \s \in K \bigr\}.$$
\end{definition}

The concept of links may be extended from vertices to simplices. The link of a simplex $\s \in K$ is the subcomplex $\dsp \text{link}_K(\s) = \bigl\{\t\in K: \t \cap \s = \emptyset, \: \t \cup \s \in K\bigr\}$. A notion of a \emph{cone} within a simplicial complex is described as follows.

\begin{definition}[Simplicial cone]
Let $K$ be a simplicial complex and $w$ be a vertex not belonging to the vertex set of $K$. The cone with apex $w$ and base $K$ is the simplicial complex formed by the join of $K$ with the vertex $w$, described as follows:
$$w*K = K \cup \bigl\{\s\cup \{w\}: \s \in K\bigr\}.$$
\end{definition}

In the theory of homotopy types of simplicial complexes, there are certain notions of elementary transformations that do not alter the homotopy type of a simplicial complex. One of these is known as \emph{simplicial collapse}, which was introduced by J.H.C. Whitehead in the late
1930s \cite{whitehead}. A condition for a simple simplicial collapse is \emph{vertex domination}, which is defined below.

\begin{definition}[Vertex domination]
A vertex $v$ of a simplicial complex $K$
is called \emph{dominated} if its link is a simplicial cone. That is, $v$ is \emph{dominated} by another vertex $w$ if each simplex $\s \in K$ containing $v$ is such that $\s\cup \{w\} \in K$. That is, link$_K(v)$ is a cone with apex $w$.
\end{definition}

\begin{definition}[Barmak and Minian\cite{adams2}]
Let $K$ be a simplicial complex, and $v\in K$ be a vertex. Let $K\setminus v$ denote the full subcomplex of $K$ spanned by the vertices different from $v$. We say that there is an \emph{elementary strong collapse} $K$ to $K\setminus v$ if the link $\text{link}_K(v)$ is a simplicial cone $w*L$. There is a \emph{strong collapse} from a complex $K$ to a subcomplex $L$ if there exists a sequence of elementary strong collapses that start at $K$ and end at $L$. The complexes $K$ and $L$ possess the same \emph{strong homotopy type}.\end{definition}

If vertex $v \in K$ is dominated, then $K \simeq K\setminus \{v\}$ because the link of $v$ is contractible. Indeed, there exists a deformation retraction from $K$ to $K\setminus \{v\}$ which maps $v$ to $w$ and also achieves a simplicial collapse of $K$ to $K\setminus \{v\}$ \cite {adams1}. These elimination operations are referred to by different terms in the literature, such as folds, elementary strong collapses, and LC reductions \cite{adams2,adams3,adams4}. 

\subsection{Special case: Vietoris-Rips complexes}
Now let $K:=VR(S,f,r)$ be a Vietoris-Rips complex arising from a point sample in space $S$ with measurement function $f: S\rightarrow \R^d$ with proximity value $r>0$. Suppose $k$ points in this sample have a coincident image in $\R^d$ by the measurement function $f$. That is, suppose there exists $s_1,\: s_2,\: ..., \: s_k \in S$ with $f(s_i)=v_0$, $i: 1\leq i\leq k$, for some vertex $v_0$ in $\R^d$. The relation defined by $x\sim y$ if and only if $f(x)=f(y)$ for two points $x$ and $y$ in $S$ is an equivalence relation whose equivalence classes are precisely all collections of points in $S$ with distinct vertices in $\R^d$. The resulting quotient space is written $S/\sim$ and has cardinality less than $S$. In constructing $S/\sim$, we ``divide'' by $\sim$ in the sense that we identify two points in $S$ if and only if they yield coincident vertices in $\R^d$. Hence, consider the Vietoris-Rips complex, $VR(S/\sim,f,r)$ built from $S/\sim$, which takes the $k$ coincident vertices as just one vertex.

More generally, consider the case where there are $m$ such clusters of repeated points so that the $i^{th}$ cluster has $k_i$ points in $S$ coincident to some point in $\R^d$, $1\leq i \leq m$.  

\begin{proposition}\label{complex2}
The Vietoris-Rips complexes $VR(S,f,r)$ and $VR(S/\sim,f,r)$ are homotopy equivalent. That is,
\begin{eqnarray*}{}
VR(S,f,r) \simeq  VR(S/\sim,f,r), \text{ where } x\sim y \text{ if } f(x)=f(y), \text { for any } x, \:y\: \in S.
\end{eqnarray*}
\end{proposition}

\begin{proof}
$ $ \newline
Use induction on the number $m$ of clusters in $S$ with a repeated vertex in $\R^d$. Let $\sim_i$ denote the collapse of the $i^{th}$ cluster in $S$ to a single vertex. Note that for the $i^{th}$  cluster, the $k_i$ coincident vertices $s_{i_j}$, $1\leq j\leq k_i$, are such that $f(s_{i_j})$ are at Euclidean distance $0$ from each other and so $s_{i_j}$ and $s_{i_l}$ are connected by an edge for $j\neq l$. Moreover, any other vertex $w$ in $S$ is either connected to all $s_{i_j}$ or connected to none.  Hence, $s_{i_1}$ dominates $s_{i_2}$, and the link of $s_{i_2}$ is a simplicial cone with apex $s_{i_1}$, which is contractible. Therefore, there is a simplicial collapse of $K$ to $K\setminus \{s_{i_2}\}$ without changing the homotopy type. A similar simplicial collapse can be done using $s_{i_3}$. Continuing in this manner, we would have removed all vertices $s_{i_j}$ for $j: 1 \leq j \leq k_i$, to achieve $VR(S/\sim_i,f,r)$ without changing the homotopy type, leaving single vertices $s_{i_1}$ at the position of the $i^{th}$ cluster. By induction, this process yields $VR(S/\sim,f,r)$. 
\end{proof}

Essentially, if two points lie on top of each other in a point cloud, the resulting Vietoris-Rips complex (with any given threshold value) will be homotopy equivalent to the simplicial complex where both points are collapsed to one. This is so because both coincident points, say, $s_1$ and $s_2$, participate in the same simplices in the complex, and a $1$-simplex (homotopy equivalent to a point) exists between $s_1$ and $s_2$.  %INSERT FIGURE HERE! Insert photo here. See notes from Dec 13, 2022.
This phenomenon carries on to higher dimensions. If $k$ points are coincident, a (standard) $k$-simplex (which is homotopy equivalent to a point) exists between them, and all $k$ simplices participate in the same simplices in the filtration. In other words, any other point is a neighbor to one of these $k$ coincident points if and only if it is also a neighbor to all the other $k-1$.

\subsection{When monotonicity fails}
If the measurement functions fail to be injective, then the positive-definiteness axiom (see lemma \ref{2}) of the simplicial Hausdorff distance, $\delta$, is not satisfied. Consequently, $\delta$ becomes a \emph{pseudometric} instead. With a pseudometric, it is possible to obtain a distance of zero for elements that are distinct from each other. Although elements with distance $0$ are considered equivalent, a pseudometric would not distinguish between simplicial complexes that may be distinct yet topologically equivalent, as in the case of $VR(S,f,r)$ and $VR(S/\sim,f,r)$ discussed in the previous section \ref{2.6.1}. In essence, even though a pseudometric will distinguish between simplicial complexes with different homotopy types, a pseudometric could consider non-isomorphic simplicial complexes to be the same. Topological information remains intact despite this loss of precision.

\subsection{Summary and future directions}

This article introduces a new notion of a simplicial Hausdorff distance and delves into computational aspects, showing its algorithmic complexity and demonstrating a version for filtered complexes. Moreover, it examines the implications of having measurement functions that are not injective, which can lead to pseudometric properties where distinct elements might have a zero distance when not topologically distinguishable.

Future directions include improving computational complexity of the method, and investigating a stability result for simplicial complexes in $\X^d$ that used the simplicial Haudorff distance to quantify robustness. Another consequent pursuit is the integration of these methods with other advanced statistical and topological deep learning techniques. By combining TDA with methods such as clustering, classification, and neural networks, we can potentially uncover deeper insights from complex data structures and improve the robustness of topological feature extraction.

Researchers are invited to apply this distance\footnote{\url{{https://github.com/NkechiC/Modified-Haudorff-Distance}}} to their simplicial complex data sets. This code is written in \emph{R} and designed to be compatible with the \emph{TDA} package (v 1.9, Fasy et al 2023\cite{rtda}) in \emph{R}, which can generate representative homology cycles as part of the output of a persistent homology routine on a point cloud data set.

\newpage
\printbibliography

@inbook{chazal,
      author = {Fr\'{e}d\'{e}ric Chazal},
      title = {High-{D}imensional {T}opological {D}ata {A}nalysis},
      publisher = {Taylor and Francis group},
      year = {2017},
      chapter = {25},
      isbn={9781315119601},
      booktitle={Handbook of Discrete and Computational Geometry}
  }

@article {hausdorff2,
    AUTHOR = {Jungeblut, Paul and Kleist, Linda and Miltzow, Tillmann},
     TITLE = {The {C}omplexity of the {H}ausdorff {D}istance},
  JOURNAL = {Discrete \& Computational Geometry. An International Journal of Mathematics and Computer Science},
    VOLUME = {71},
      YEAR = {2024},
    NUMBER = {1},
     PAGES = {177--213},
      ISSN = {0179-5376,1432-0444},
   MRCLASS = {68 (14Q20 14Q30)},
  MRNUMBER = {4685713},
       DOI = {10.1007/s00454-023-00562-5}
}

@book{hausdorff,
author = {Hausdorff, Felix},
title ={Grundz\"{u}ge {D}er {M}engenlehre},
publisher={Von Veit \& Company},
address={Leipzig, Germany},
year={1914}
}

@phdthesis{godau,
    author = {Godau, Michael},
    title = {On the complexity of measuring the similarity between geometric objects in higher dimensions},
    school = {Freie Universit\"{a}t Berlin},
    year = {1999}
}

@incollection{altbook,
    author = {Helmut Alt and Peter Braß and Michael Godau and Christian Knauer and Carola Wenk},
    editor={Boris Aronov and Saugata Basu and J\'{a}nos Pach and Micha Sharir},
    title = {Computing the {H}ausdorff {D}istance of {G}eometric {P}atterns and {S}hapes},
    booktitle = {Discrete and {C}omputational {G}eometry: The {G}oodman-{P}ollack {F}estschrift},
    publisher = {Springer} ,
    year = {2003},
    volume ={25}, 
    series={Algorithms and Combinatorics},
    pages= {65–76},
    doi={doi.org/10.1007/978-3-642-55566-4}
}

@article{whitehead,
    author = {Whitehead, J. H. C.},
    title = "{Simplicial Spaces, Nuclei and m-Groups}",
    journal = {Proceedings of the London Mathematical Society},
    volume = {s2-45},
    number = {1},
    pages = {243-327},
    year = {1939},
    issn = {0024-6115},
    doi = {10.1112/plms/s2-45.1.243}
}

@article{adams2,
    author = {Barmak, J.A. and Minian, E.G. },
    title = {Strong Homotopy Types, Nerves and Collapses},
    journal = {Discrete \& Computational Geometry},
    year = {2012},
    pages = {301–328},
    volume = {47},
    doi={doi.org/10.1007/s00454-011-9357-5}
}

@article{adams3,
    author ={Babson, E. and Kozlov, D.N.} ,
    title = {Complexes of graph homomorphisms},
    journal = {Israel Journal of Mathematics},
    year = {2006},
    volume={152},
    pages={285–312},
    doi={doi.org/10.1007/BF02771988}
}

@article{adams4,
    author = {Ji\v{r}\'{i} Matou\v{s}ek} ,
    title = {L{C} {R}eductions {Y}ield {I}somorphic {S}implicial {C}omplexes},
    journal = {Contributions to Discrete Mathematics},
    year = {2008},
    issn = {1715-0868},
    volume = {3},
    number ={2},
    publisher ={University of Calgary}
}

@article{adams1,
    author = {Micha\cancel{l} Adamaszek and Henry Adams and Florian Frick and Chris Peterson and Corrine Previte-Johnson},
    title = {Nerve Complexes of Circular Arcs},
    journal = {Discrete \& Computational Geometry} ,
    year = {2016},
    doi={10.1007/s00454-016-9803-5},
    volume={56},
    pages = {251–273},
    publisher ={Springer Science},
    location = {New York}
}

@book{baby1,
    author = {Charles E. Chidume and Chukwudi O. Chidume},
    title = {Foundations of Mathematical Analysis},
    publisher = {Ibadan University Press},
    year = {2013},
address ={Ibadan, Nigeria},
isbn={978-978-8456-32-2}
}

@book{munkres,
  title={Topology: Pearson New International Edition},
  author={Munkres, J.R.},
  isbn={9781292036786},
  year={2013},
  publisher={Pearson Education}
}

@article{wang1,
author= {Alt, Helmut and Behrends, Bernd and Bl\"{o}mer, Johannes},
year={1995},
title={Approximate matching of polygonal shapes},
journal={Annals of Mathematics and Artificial Intelligence},
pages= {251-265},
volume={13},
doi={10.1007/BF01530830}
}

@Manual{r,
    title = {R: A Language and Environment for Statistical Computing},
    author = {{R Core Team}},
    organization = {R Foundation for Statistical Computing},
    address = {Vienna, Austria},
    year = {2022},
    url = {https://www.R-project.org/},
  }

@Manual{rtda,
    title = {TDA: Statistical Tools for Topological Data Analysis},
    author = {Brittany T. Fasy and Jisu Kim and Fabrizio Lecci and Clement Maria and David L. Millman and Vincent Rouvreau.},
    year = {2023},
    note = {R package version 1.9},
    url = {https://CRAN.R-project.org/package=TDA},
  }
\end{document}